\documentclass[a4paper,12pt]{article}
\usepackage{mathtools,amssymb,amsthm}
\usepackage{parskip}
\usepackage{url}
\usepackage[hidelinks]{hyperref}
\usepackage[style=authoryear]{biblatex}
\newtheorem{theorem}{Theorem}[section]

\newtheorem{lemma}[theorem]{Lemma}

\theoremstyle{definition}
\newtheorem{definition}[theorem]{Definition}
\newtheorem{example}[theorem]{Example}

\theoremstyle{remark}

\newtheorem*{problem}{Problem}

\newcommand{\fset}[2]{\mathcal F^{(#1)}_{#2}}
\newcommand{\N}{\mathbb N}
\DeclareMathOperator{\atom}{atom}
\DeclareMathOperator{\ran}{ran}

\title{Tuples as sets}
\author{Adrian Ducourtial}
\date{June 16\textsuperscript{th}, 2026}

\begin{document}

\maketitle

\section{Introduction}

Fix an infinite domain $D_0$ and recursively define $D_{n+1}$ as the set of finite non-empty subsets of $D_n$ for $n \geq 0$. For $k \geq 1$, $n \geq 0$, we define the set $\fset kn$, whose elements are functions $D_0^k \to D_n$, by recursion on $n$, as follows.
\begin{enumerate}
    \item $\fset k0$ is the set of $k$-ary projections on $D_0$, i.e., those functions $D_0^k \to D_0$ that are of the form $(x_1, \dots, x_k) \mapsto x_i$ for some $1 \leq i \leq k$.
    \item For $n \geq 0$, $\fset k{n+1}$ is the set of those functions $D_0^k \to D_{n+1}$ that are of the form $\mathbf x \mapsto \{g(\mathbf x) : g \in \mathcal G\}$ for some non-empty $\mathcal G \subseteq \fset kn$.
\end{enumerate}
\begin{example}\label{ex}
    The projections $(x_1, x_2) \mapsto x_1$ and $(x_1, x_2) \mapsto x_2$ on $D_0$ are in $\fset 20$. Therefore, the functions $(x_1, x_2) \mapsto \{x_1\}$ and $(x_1, x_2) \mapsto \{x_1, x_2\}$ of type $D_0^2 \to D_1$ are in $\fset 21$. It follows that the function $f_K : D_0^2 \to D_2 : (x_1, x_2) \to \{\{x_1\}, \{x_1, x_2\}\}$ is in $\fset 22$.
\end{example}
The function $f_K$ of Example~\ref{ex}, originally considered by \textcite[171]{Kur}, is injective. This allows us to identify the ordered pair $(x_1, x_2)$ of objects $x_1, x_2 \in D_0$ with the set $f_K(x_1, x_2) \in D_2$; for then, the injectivity of $f_K$ amounts to the characteristic property of ordered pairs (of objects in $D_0$). This identification is standard in classical set theory \parencite[291]{Kan}, where, generalizing to classes, $D_0$ is the universe of sets and $D_n \subseteq D_0$ for all $n$.

More generally, for $k \geq 1$, $n \geq 1$, an injection $f \in \fset kn$ provides a definition of $k$-tuples over $D_0$ as sets in $D_n$. In this paper, we ask the following.

\begin{problem}
    Let $k \geq 1$, $n \geq 1$ be given. Does $\fset kn$ contain an injection?
\end{problem}

The paper is organized as follows. Section~\ref{sec:gp} develops the basic theory, leading us to restate our problem as the study of a certain integer function. Sections~\ref{sec:ql},~\ref{sec:val}, and~\ref{sec:qt} then initiate this study, giving qualitative results, function values for small arguments, and quantitative results, respectively. Finally, Section~\ref{sec:var} discusses several variants of the problem.

\section{Good pairs}\label{sec:gp}

Denote by $\N$ the set of positive integers. Call a pair $(k, n) \in \N \times \N$ \emph{good} if $\fset kn$ contains an injection.

\begin{lemma}\label{lem:var}
    Let $n \geq 0$, $k, \ell \geq 1$, $f \in \fset kn$, and $\pi_{(i)} \in \fset \ell0$ for $1 \leq i \leq k$ be given. The function $f' : D_0^\ell \to D_n$ given by
    \[f'(\mathbf x) = f(\pi_{(1)}(\mathbf x), \dots, \pi_{(k)}(\mathbf x))\]
    is in $\fset \ell n$.
\end{lemma}
\begin{proof}
    We prove the statement by induction on $n$.
    \begin{itemize}
        \item \textit{Base case.} Since $f \in \fset k0$, we have that $f = \pi_i : (x_1, \dots, x_k) \mapsto x_i$ for some $1 \leq i \leq k$. Hence,
        \[f'(\mathbf x) = \pi_i(\pi_{(1)}(\mathbf x), \dots, \pi_{(k)}(\mathbf x)) = \pi_{(i)}(\mathbf x)\]
        for $\mathbf x \in D_0^\ell$. In other words, $f' = \pi_{(i)} \in \fset \ell0$, and we are done.
        \item \textit{Inductive step.} Let $n \geq 0$ be given. The function $f \in \fset k{n+1}$ is given by $f(\mathbf x) = \{g(\mathbf x) : g \in \mathcal G\}$ for some non-empty $\mathcal G \subseteq \fset kn$. For each $g \in \mathcal G$, the function $g' : D_0^\ell \to D_n$ given by
        \[g'(\mathbf x) = g(\pi_{(1)}(\mathbf x), \dots, \pi_{(k)}(\mathbf x))\]
        is in $\fset \ell n$ by the inductive hypothesis. Together, this entails that
        \begin{align*}
            f'(\mathbf x) &= f(\pi_{(1)}(\mathbf x), \dots, \pi_{(k)}(\mathbf x))\\
            &= \{g(\pi_{(1)}(\mathbf x), \dots, \pi_{(k)}(\mathbf x)) : g \in \mathcal G\}\\
            &= \{g'(\mathbf x) : g \in \mathcal G\}
        \end{align*}
        for $\mathbf x \in D_0^\ell$. In other words, $f'$ is determined by the set $\{g' : g \in \mathcal G\} \subseteq \fset \ell n$, and since this set is non-empty, we have $f \in \fset \ell{n+1}$ as desired.
    \end{itemize}
    This concludes the proof.
\end{proof}

\begin{definition}
    For $n \geq 0$, recursively define $\iota_n \in \fset 1n$ by letting $\iota_0 : x \mapsto x$ and $\iota_{n+1} : x \mapsto \{\iota_n(x)\}$ for all $n$.
\end{definition}

\begin{lemma}\label{lem:iota}
    For $n \geq 0$, the function $\iota_n$ is injective.
\end{lemma}
\begin{proof}
    By induction on $n$.
    \begin{itemize}
        \item \textit{Base case.} The function $\iota_0$ is trivially injective.
        \item \textit{Inductive step.} For $x, y \in D_0$, we have
        \begin{align*}
            \iota_{n+1}(x) = \iota_{n+1}(y) &\iff \{\iota_n(x)\} = \{\iota_n(y)\}\\
            &\iff \iota_n(x) = \iota_n(y)\\
            &\iff x = y &&\text{by ind.\ hyp.},
        \end{align*}
        as desired.
    \end{itemize}
    This concludes the proof.
\end{proof}

\begin{definition}
    Let $\sigma$ be a permutation on $D_0$, i.e., a bijection $D_0 \to D_0$. We define the permutation $\hat\sigma_n$ on $D_n$ by recursion on $n \geq 0$. Let $\hat\sigma_0 = \sigma$, and for $n \geq 0$, let $\hat\sigma_{n+1}$ be given by $\hat\sigma_{n+1}(S) = \{\hat\sigma_n(s) : s \in S\}$.
\end{definition}

\begin{lemma}\label{lem:permcomm}
    Let $k \geq 1$, $n \geq 0$, and $f \in \fset kn$ be given, and let $\sigma : D_0 \to D_0$ be a permutation. For $\mathbf x = (x_1, \dots, x_k) \in D_0^k$, we have
    \[f(\sigma\mathbf x) = \hat\sigma_n(f(\mathbf x)),\]
    where $\sigma\mathbf x = (\sigma(x_1), \dots, \sigma(x_k))$.
\end{lemma}
\begin{proof}
    We prove the statement by induction on $n$.
    \begin{itemize}
        \item \textit{Base case.} Since $f \in \fset k0$, we have $f = \pi_i : (x_1, \dots, x_k) \mapsto x_i$ for some $1 \leq i \leq k$. Therefore,
        \[f(\sigma\mathbf x) = \sigma(x_i) = \sigma(f(\mathbf x))\]
        for $\mathbf x \in D_0^k$, and since $\sigma = \hat\sigma_0$, we are done.
        \item \textit{Inductive step.} The function $f$ is given by $f(\mathbf x) = \{g(\mathbf x) : g \in \mathcal G\}$ for some non-empty $\mathcal G \subseteq \fset kn$. Hence,
        \begin{align*}
            f(\sigma\mathbf x) &= \{g(\sigma\mathbf x) : g \in \mathcal G\}\\
            &= \{\hat\sigma_n(g(\mathbf x)) : g \in \mathcal G\} &&\text{by ind.\ hyp.}\\
            &= \hat\sigma_{n+1}(f(\mathbf x)) &&\text{by def.\ of } \hat\sigma_{n+1}
        \end{align*}
        for $\mathbf x \in D_0^k$, as desired.
    \end{itemize}
    This concludes the proof.
\end{proof}

\begin{definition}
    We define the function $\atom_n : D_n \to D_1$ by recursion on $n \geq 0$. Let $\atom_0 : x \mapsto \{x\}$ and let $\atom_{n+1} : S \mapsto \bigcup_{s\in S} \atom_n(s)$ for $n \geq 0$.
\end{definition}

\begin{lemma}\label{lem:atom}
    Let $k \geq 1$, $n \geq 0$, and $f \in \fset kn$ be given. For $\mathbf x \in D_0^k$, writing $\mathbf x = (x_1, \dots, x_k)$, we have
    \[\atom_n(f(\mathbf x)) \subseteq \{x_1, \dots, x_k\}.\]
\end{lemma}
\begin{proof}
    By induction on $n$.
    \begin{itemize}
        \item \textit{Base case.} We have $f = \pi_i : (x_1, \dots, x_k) \mapsto x_i$ for some $1 \leq i \leq k$. Hence, $\atom_0(f(\mathbf x)) = \atom_0(x_i) = \{x_i\} \subseteq \{x_1, \dots, x_k\}$, as desired.
        \item \textit{Inductive step.} The function $f$ is given by $f(\mathbf x) = \{g(\mathbf x) : g \in \mathcal G\}$ for some non-empty $\mathcal G \subseteq \fset kn$. Therefore,
        \[\atom_{n+1}(f(\mathbf x)) = \bigcup_{g \in \mathcal G} \atom_n(g(\mathbf x)).\]
        By the inductive hypothesis, $\atom_n(g(\mathbf x)) \subseteq \{x_1, \dots, x_k\}$ for every $g \in \mathcal G$, whence $\atom_{n+1}(f(\mathbf x)) \subseteq \{x_1, \dots, x_k\}$ as desired.
    \end{itemize}
    This concludes the proof.
\end{proof}

\begin{lemma}\label{lem:atomiota}
    We have $\atom_n(\iota_n(x)) = \{x\}$ for all $n \geq 0$.
\end{lemma}
\begin{proof}
    By induction on $n$.
    \begin{itemize}
        \item \textit{Base case.} We have $\atom_0(\iota_0(x)) = \atom_0(x) = \{x\}$ as desired.
        \item \textit{Inductive step.} We have
        \[\atom_{n+1}(\iota_{n+1}(x)) = \atom_{n+1}(\{\iota_n(x)\}) = \atom_n(\iota_n(x)),\]
        which is equal to $\{x\}$ by the inductive hypothesis.
    \end{itemize}
    This concludes the proof.
\end{proof}

\begin{theorem}\label{thm:diag}
    Let $k, n \in \N$ be given. If $(k, n)$ is good, then $(k+1, n+1)$ is good.
\end{theorem}
\begin{proof}
    Suppose some $f \in \fset kn$ is injective. We have that $f$ is given by $f(\mathbf x) = \{g(\mathbf x) : g \in \mathcal G\}$ for some non-empty $\mathcal G \subseteq \fset k{n-1}$.

    Consider the function $f' : D_0^{k+1} \to D_{n+1}$ given by
    \begin{align*}
        f'(x_1, \dots, x_{k+1})\,=\ &\big\{\left\{g(x_1, \dots, x_k)\right\} : g \in \mathcal G\big\}\ \cup\\
        &\big\{\left\{g(x_1, \dots, x_k), \iota_{n-1}(x_{k+1})\right\} : g \in \mathcal G \big\}.
    \end{align*}
    By use of Lemma~\ref{lem:var}, it can be seen that $f' \in \fset {k+1}{n+1}$. We will show that $f'$ is injective.

    Let $\mathbf x, \mathbf y \in D_0^{k+1}$ be given. Write $\mathbf x = (x_1, \dots, x_{k+1})$, $y = (y_1, \dots, y_{k+1})$, and suppose $f'(\mathbf x) = f'(\mathbf y)$. We wish to show that $\mathbf x = \mathbf y$, i.e., that $x_i = y_i$ for all $1 \leq i \leq k+1$.

    Notice that, for all $g \in \mathcal G$, we have $\{g(x_1, \dots, x_k)\} \in f'(\mathbf x)$, and this set is a singleton. We claim that the converse holds as well: If $s \in f'(\mathbf x)$ is a singleton, then $s = \{g(x_1, \dots, x_k)\}$ for some $g \in \mathcal G$. To see this, note that $s$ is either of the form $\{g(x_1, \dots, x_k)\}$ for some $g \in \mathcal G$ (in which case we are done) or of the form $\{g(x_1, \dots, x_k), \iota_{n-1}(x_{k+1})\}$ for some $g \in \mathcal G$. In the latter case, since $s$ is a singleton, we have $g(x_1, \dots, x_k) = \iota_{n-1}(x_{k+1})$, and hence $s = \{g(x_1, \dots, x_k)\}$ as desired. An analogous argument shows that the singletons in $f'(\mathbf y)$ are exactly the sets in $D_n$ that are of the form $\{g(y_1, \dots, y_k)\}$ for some $g \in \mathcal G$.

    Denoting by $S_{\mathbf x}$ (respectively, $S_{\mathbf y}$) the set of singletons in $f'(\mathbf x)$ (respectively, $f'(\mathbf y)$), we thus have
    \begin{align*}
        \bigcup S_{\mathbf x} &= f(x_1, \dots, x_k),\\
        \bigcup S_{\mathbf y} &= f(y_1, \dots, y_k).
    \end{align*}
    Since $f'(\mathbf x) = f'(\mathbf y)$, we have $S_{\mathbf x} = S_{\mathbf y}$ by definition. Therefore, $\bigcup S_{\mathbf x} = \bigcup S_{\mathbf y}$, i.e., $f(x_1, \dots, x_k) = f(y_1, \dots, y_k)$; and since $f$ is injective, this means that $x_i = y_i$ for $1 \leq i \leq k$. It remains to show that $x_{k+1} = y_{k+1}$.

    Writing $S = S_{\mathbf x} = S_{\mathbf y}$, define
    \begin{align*}
        T_{\mathbf x} &= f'(\mathbf x) \setminus S,\\
        T_{\mathbf y} &= f'(\mathbf y) \setminus S.
    \end{align*}
    Note that the elements of $T_{\mathbf x}$ are exactly those sets in $D_n$ that are of the form $\{g(x_1, \dots, x_k), \iota_{n-1}(x_{k+1})\}$ for some $g \in \mathcal G$ with $g(x_1, \dots, x_k) \neq \iota_{n-1}(x_{k+1})$; similarly, \emph{mutatis mutandis}, for $T_{\mathbf y}$. Moreover, since $f'(\mathbf x) = f'(\mathbf y)$, we have $T_{\mathbf x} = T_{\mathbf y}$. We now split into cases.
    \begin{itemize}
        \item Suppose $T_{\mathbf x} = T_{\mathbf y}$ is empty. This means that for every $g \in \mathcal G$, it is the case that $g(x_1, \dots, x_k) = \iota_{n-1}(x_{k+1})$. Therefore, $f(x_1, \dots, x_k) = \{\iota_{n-1}(x_{k+1})\} = \iota_n(x_{k+1})$. A similar argument shows $f(y_1, \dots, y_k) = \iota_n(y_{k+1})$. Since $f(x_1, \dots, x_k) = f(y_1, \dots, y_k)$, it follows that $\iota_n(x_{k+1}) = \iota_n(y_{k+1})$; hence, by Lemma~\ref{lem:iota}, we have $x_{k+1} = y_{k+1}$ as desired.
        \item Otherwise, let $t \in T_{\mathbf x} = T_{\mathbf y}$ be given. We seek to derive a contradiction. Notice that every member of $T_{\mathbf x}$ contains $\iota_{n-1}(x_{k+1})$; likewise, every member of $T_{\mathbf y}$ contains $\iota_{n-1}(y_{k+1})$. Therefore, in particular, $\iota_{n-1}(x_{k+1}), \iota_{n-1}(y_{k+1}) \in t$. Since $t \notin S$, it must be the case that $t$ has exactly two elements; therefore, $t = \{\iota_{n-1}(x_{k+1}), \iota_{n-1}(y_{k+1})\}$ with $\iota_{n-1}(x_{k+1}) \neq \iota_{n-1}(y_{k+1})$. Notice, however, that $t$ was chosen arbitrarily; and so
        \[T_{\mathbf x} = T_{\mathbf y} = \big\{\left\{\iota_{n-1}(x_{k+1}), \iota_{n-1}(y_{k+1})\right\}\big\}.\]
        This implies, on the one hand, that $\iota_{n-1}(y_{k+1})$ is the unique member of $f(x_1, \dots, x_k)$ that is distinct from $\iota_{n-1}(x_{k+1})$, and, on the other, that $\iota_{n-1}(x_{k+1})$ is the unique member of $f(y_1, \dots, y_k)$ that is distinct from $\iota_{n-1}(y_{k+1})$. As previously established, $f(x_1, \dots, x_k) = f(y_1, \dots, y_k)$, and so together, we have
        \[f(x_1, \dots, x_k) = f(y_1, \dots, y_k) = \{\iota_{n-1}(x_{k+1}), \iota_{n-1}(y_{k+1})\}.\]
        Now, let $\sigma : D_0 \to D_0$ be the permutation $(x_{k+1}\ y_{k+1})$. We have
        \begin{align*}
            f(\sigma(x_1), \dots, \sigma(x_k)) &= \hat\sigma_n(f(x_1, \dots, x_k)) &&\text{by Lemma~\ref{lem:permcomm}}\\
            &= \hat\sigma_n\big(\{\iota_{n-1}(x_{k+1}), \iota_{n-1}(y_{k+1})\}\big)\\
            &= \{\hat\sigma_{n-1}(\iota_{n-1}(x_{k+1})), \hat\sigma_{n-1}(\iota_{n-1}(y_{k+1}))\} &&\text{by def.\ of }\hat\sigma_n\\
            &= \{\iota_{n-1}(\sigma(x_{k+1})), \iota_{n-1}(\sigma(y_{k+1}))\} &&\text{by Lemma~\ref{lem:permcomm}}\\
            &= \{\iota_{n-1}(y_{k+1}), \iota_{n-1}(x_{k+1})\}\\
            &= f(x_1, \dots, x_k).
        \end{align*}
        Since $f$ is injective, this means that $\sigma(x_i) = x_i$ for all $1 \leq i \leq k$.

        Now, because
        \begin{align*}
            \atom_n(f(x_1, \dots, x_k)) &= \atom_n\!\big(\{\iota_{n-1}(x_{k+1}), \iota_{n-1}(y_{k+1})\}\big)\\
            &\supseteq \atom_{n-1}(\iota_{n-1}(x_{k+1}))\\
            &= \{x_{k+1}\} &&\text{by Lemma~\ref{lem:atomiota}},
        \end{align*}
        it follows by Lemma~\ref{lem:atom} that $x_{k+1} \in \{x_1, \dots, x_k\}$, i.e., $x_{k+1} = x_j$ for some $1 \leq j \leq k$; and since $x_j = \sigma(x_j)$, this means $x_{k+1} = \sigma(x_{k+1}) = y_{k+1}$. However, we established earlier that $\iota_{n-1}(x_{k+1}) \neq \iota_{n-1}(y_{k+1})$, implying that $x_{k+1} \neq y_{k+1}$, a contradiction.
    \end{itemize}
    This concludes the proof.
\end{proof}

\begin{theorem}\label{thm:left}
    Let $k, n \in \N$ be given. If $(k+1, n)$ is good, then $(k, n)$ is good.
\end{theorem}
\begin{proof}
    Suppose some $f \in \fset {k+1}n$ is injective. Then the function $f' : D_0^k \to D_n$ given by
    \[f'(x_1, \dots, x_k) = f(x_1, \dots, x_k, x_k)\]
    is in $\fset kn$ by Lemma~\ref{lem:var} and is clearly injective.
\end{proof}

\begin{theorem}\label{thm:up}
    Let $k, n \in \N$ be given. If $(k, n)$ is good, then $(k, n+1)$ is good.
\end{theorem}
\begin{proof}[First proof]
    If $(k, n)$ is good, then $(k+1, n+1)$ is good by Theorem~\ref{thm:diag}, whence $(k, n+1)$ is good by Theorem~\ref{thm:left}.
\end{proof}
\begin{proof}[Second proof]
    Suppose $f \in \fset kn$ is injective. Then the function $f' \in \fset k{n+1}$ given by $f'(\mathbf x) = \{f(\mathbf x)\}$ is easily seen to be injective.
\end{proof}

\begin{lemma}\label{lem:11}
    The pair $(1, 1)$ is good.
\end{lemma}
\begin{proof}
    The function $f \in \fset 11$ given by $f(x_1) = \{x_1\}$ is injective.
\end{proof}

\begin{theorem}\label{thm:ex}
    For every $k \in \N$, there is $n \in \N$ such that $(k, n)$ is good. 
\end{theorem}
\begin{proof}
    By induction on $k$. In the base case, apply Lemma~\ref{lem:11}; for the inductive step, apply Theorem~\ref{thm:diag}.
\end{proof}

We are led to consider the following.

\begin{definition}
    The function $\mu : \N \to \N$ sends $k$ to the least $n$ such that $(k, n)$ is good, for all $k$.
\end{definition}

By Theorem~\ref{thm:ex}, $\mu(k)$ is well-defined for all $k \in \N$; and by Theorem~\ref{thm:up}, $(k, n)$ is good just in case $n \geq \mu(k)$ for all $k, n \in \N$. Thus, the course-of-values of $\mu$ provides complete information as to which pairs $(k, n)$ are good. Until Section~\ref{sec:var}, we will henceforth devote our attention to the study of $\mu$.

\section{Qualitative results}\label{sec:ql}

In this section, we establish a couple of qualitative results concerning $\mu$.

\begin{theorem}\label{thm:mon}
    The function $\mu$ is nondecreasing.
\end{theorem}
\begin{proof}
    Suppose, aiming for a contradiction, that $\mu(k) > \mu(\ell)$ for some $k < \ell$. By Theorem~\ref{thm:left}, it follows that $(k, \mu(\ell))$ is good, whence $\mu(\ell) \leq \mu(k)$, a contradiction.
\end{proof}

\begin{theorem}\label{thm:unb}
    The function $\mu$ is unbounded.
\end{theorem}
\begin{proof}
    Recall that $D_0$ is infinite. Let $h \geq 2$, $a_1, \dots, a_h \in D_0$ be given with $a_i \neq a_j$ for $i \neq j$. Define $D_0' = \{a_1, \dots, a_h\}$, and for $n \geq 0$, recursively define $D_{n+1}'$ to be the set of non-empty subsets of $D_n'$. Note that $D_n' \subseteq D_n$ and $D_n'$ is finite for all $n$; moreover, the value $|D_n'|$ is strictly increasing in $n$.

    Now, let $k \in \N$ be given. We wish to show that $\mu(\ell) > \mu(k)$ for some $\ell > k$. Writing $n = \mu(k)$, let $f \in \fset kn$ be injective. Consider the restriction $f'$ of $f$ to $(D_0')^k$. By Lemma~\ref{lem:atom}, we have $\atom_n f'(\mathbf x) \subseteq D_0'$ for $\mathbf x \in (D_0')^k$, whence $f'(\mathbf x) \in D_n'$ for all $\mathbf x$. Since $f$ is injective, so is $f'$, implying that $|D_n'| \geq |(D_0')^k|$. But since $|D_0| = h \geq 2$, we may choose $\ell > k$ so that $|(D_0')^\ell| > |D_n'|$. Write $m = \mu(\ell)$, let $g \in \fset \ell m$ be injective, and consider the restriction $g'$ of $g$ to $(D_0')^\ell$. Analogously to $f'$, we have $\ran(g') \subseteq D_m'$, whence $|D_m'| \geq |(D_0')^\ell|$ by injectivity of $g'$. Together, this gives $|D'_m| > |D'_n|$, implying that $m > n$, i.e., $\mu(\ell) > \mu(k)$, as desired.
\end{proof}

\section{Determination of values}\label{sec:val}

In this section, we determine $\mu(k)$ exactly for $k \leq 5$. This will involve the use of three programs: one for enumerating the members of $\fset kn$, one for deciding whether some $f \in \fset kn$ is injective, and one for proposing possible injections $f \in \fset kn$, all for given $k, n \in \N$. The programs, as well as the theory behind their design, are available in a repository.\footnote{See \url{https://codeberg.org/ducourtial/tuples}.\label{fn}}

\begin{theorem}
    The following hold.
    \begin{itemize}
        \item[\textnormal{(i)}] $\mu(1) = 1$.
        \item[\textnormal{(ii)}] $\mu(2) = 2$.
        \item[\textnormal{(iii)}] $\mu(3) = \mu(4) = \mu(5) = 3$.
    \end{itemize}
\end{theorem}
\begin{proof}
    (i)\enspace It suffices to show $1 \leq \mu(1) \leq 1$. The first inequality is trivial, while the second follows from Lemma~\ref{lem:11}.

    (ii)\enspace It is easy to show that no function in $\fset 21$ is injective, whence $\mu(2) > 1$; on the other hand, $f_K \in \fset 22$ from Example~\ref{ex} is injective, so $\mu(2) \leq 2$.

    (iii)\enspace It suffices to show that $\mu(k) > 2$ for $k \geq 3$ while $\mu(k) \leq 3$ for $k \leq 5$, which, by Theorem~\ref{thm:mon}, amounts to showing that $\mu(3) > 2$ while $\mu(5) \leq 3$. For the former, an exhaustive search of $\fset 32$ yielded no injections, indicating $\mu(3) > 2$ as desired. For the latter, the function $f \in \fset 53$ given by
    \begin{align*}
        f(x_1, \dots, x_5) = \Big\{&\{\{x_1, x_2\}, \{x_1, x_2, x_4, x_5\}, \{x_1, x_3\}, \{x_1, x_4, x_5\}, \{x_5\}\},\\
        &\{\{x_1, x_4\}, \{x_3\}\},\phantom{\Big\}}\\
        &\{\{x_2, x_3\}, \{x_2, x_3, x_4\}\},\phantom{\Big\}}\\
        &\{\{x_4\}\}\Big\}
    \end{align*}
    has been verified to be injective, so that $\mu(5) \leq 3$ as desired.
\end{proof}

\section{Quantitative results}\label{sec:qt}

In this section, we mention some quantitative properties of the function $\mu$.

\begin{lemma}
    We have $\mu(k+1) \leq \mu(k)+1$ for all $k \in \N$.
\end{lemma}
\begin{proof}
    Since $(k, \mu(k))$ is good, we have that $(k+1, \mu(k)+1)$ is good by Theorem~\ref{thm:diag}, whence $\mu(k+1) \leq \mu(k)+1$ as desired.
\end{proof}

Simple bounds on $\mu$ can be given. On the one hand, the proof of Theorem~\ref{thm:unb} can be adapted to give a sublogarithmic lower bound; on the other, one can combine injections of low arity to obtain a logarithmic upper bound. Beyond this, the growth rate of $\mu$ remains to be determined.

\section{Variations on the problem}\label{sec:var}

In this section, we briefly discuss some variants of our problem. The variations may be combined; here, we consider them individually.

\textbf{The empty set.} Our definition of the $D_n$ for $n \geq 1$ excludes the empty set as a member. One may consider including it for every $n \geq 1$ and admitting functions that return it. This would make available the binary function $f_W$ given by
\[f_W(x_1, x_2) = \{\{\emptyset, \{x_1\}\}, \{\{x_2\}\}\},\]
shown to be injective by \textcite{Wie} prior to the work of Kuratowski \parencite[see][290]{Kan}. Here, the empty set serves to ``tag'' the argument $x_1$; more generally, one may study how the empty set, and derived sets such as $\{\emptyset\}$, allow for new injective functions.

\textbf{Cardinality.} Conversely, one can consider restricting $D_{n+1}$ to those non-empty subsets of $D_n$ that are of cardinality at most $c \geq 1$, for $n \geq 0$, and restricting the function classes accordingly. One may ask how the choice of $c$ affects the values of the function $\mu$, and study the special case $c = 2$.

\textbf{Atoms.} We have taken $D_0$ to be an infinite set. One may instead take $|D_0| = c$ for some fixed $c \in \N$ and study whether this decreases $\mu(k)$ for $k > c$ compared to the infinite case.

\textbf{Total orders.} Rather than having the functions in $\fset kn$ be defined on general $k$-tuples over $D_0$, we may restrict them to those tuples whose components are pairwise distinct, for $k \geq 1$, $n \geq 0$. The latter correspond, in effect, to totally ordered $n$-element subsets of $D_0$. In fact, in his \parencite*{Kur}, Kuratowski proposed a set-theoretic definition of ordering relations, which we can adapt to our context in the following manner. For $k \geq 1$, denote by $(D_0^k)^*$ the set of $k$-tuples over $D_0$ having pairwise distinct components. Then, the function $f : (D_0^k)^* \to D_2$ given by
\[f(x_1, \dots, x_k) = \{\{x_1\}, \{x_1, x_2\}, \dots, \{x_1, \dots, x_k\}\}\]
is injective for all $k$. Notice that the codomain is independent of $k$, and that for $k = 2$, we recover (in restricted form) the function $f_K$. As a historical note, although Kuratowski suggested to define ordered pairs as sets by means of $f_K$, he did not establish the injectivity of this function on general pairs, but only on pairs whose two coordinates are distinct. The former does not directly follow from the latter, however.

\textbf{Multisets.} Instead of having the members of $D_{n+1}$ be sets, we may take them to be multisets, i.e., sets to which their elements may belong ``with multiplicity'', for $n \geq 0$. Denote by $\widetilde D_1$ the set of non-empty finite multisets of objects in $D_0$. Then, for $k \geq 1$, the function $f : D_0^k \to \widetilde D_1$ given by
\[f(x_1, \dots, x_k) = [\underbrace{x_1, \dots, x_1}_{2^0}, \underbrace{x_2, \dots, x_2}_{2^1}, \dots, \underbrace{x_k, \dots, x_k}_{2^{k-1}}]\]
is injective. To see this, let $a \in D_0$ be given and consider the number of times $a$ is a member of $f(x_1, \dots, x_k)$; the binary expansion of this number tells us which $x_i$, if any, are equal to $a$.

\textbf{Cumulativity.} We have $D_{n+1} \subseteq \mathcal P(D_n)$ for all $n \geq 0$, where $\mathcal P(\cdot)$ denotes the powerset operation. One could more generally take $D_{n+1} \subseteq \mathcal P\!\left(\bigcup_{i\leq n}D_i\right)$ for $n \geq 0$. This setting, however, requires us to settle on a notion of heterogeneous equality. For instance, one will have to decide whether $D_0$ is disjoint from $D_1 \setminus D_0$. If no member of $D_0$ is a set, then this statement holds; but if some members of $D_0$ are sets, and especially if non-well-founded sets are admitted, then this may fail to be the case.

\printbibliography

\bigskip

\textbf{Technology disclosure.} This work was partially carried out with the use of text generation models, specifically in suggesting the proof strategies for Theorems~\ref{thm:diag} and~\ref{thm:unb} and in connection with coding (see footnote~\ref{fn} above). The author takes full responsibility for the contents of this paper.

\end{document}